\documentclass[reqno, dvipsnames]{amsart}
\usepackage[utf8]{inputenc}
\usepackage{amscd, amssymb, amsthm, mathrsfs, amsmath, amsrefs, amsfonts, graphicx}
\usepackage[all]{xy}
\usepackage{pifont}
\usepackage[colorlinks,linkcolor=red,citecolor=blue,urlcolor=blue,hypertexnames=false]{hyperref}
\hypersetup{
unicode=true
}
\usepackage{graphics}
\usepackage{tensor}
\usepackage{xcolor}
\usepackage{mathcommand}
\usepackage{todonotes}
\usepackage{enumitem}

\usepackage{tikz}
\usetikzlibrary{arrows.meta}

\numberwithin{equation}{section}
\newtheorem{thm}[equation]{Theorem}
\newtheorem{prop}[equation]{Proposition}
\newtheorem{coro}[equation]{Corollary}
\newtheorem{lem}[equation]{Lemma}

\theoremstyle{definition}
\newtheorem{defi}[equation]{Definition}
\newtheorem{con}[equation]{Construction}
\newtheorem{rem}[equation]{Remark}
\newtheorem{exa}[equation]{Example}

\usepackage{todonotes}

\newcommand{\Alg}{{\mathbf{Alg}}}
\newcommand{\GrAlg}{{G_{\hspace{-0.1em}\mathrm{gr}}\Alg}}
\newcommand{\Ggr}{{{G_{\hspace{-0.1em}\mathrm{gr}}}}}
\newcommand{\jgr}{{{j^G_{\hspace{-0.1em}\mathrm{gr}}}}}

\newcommand{\sSet}{\mathbf{sSet}}

\newcommand{\id}{\mathrm{id}}

\newcommand{\colim}{\operatornamewithlimits{colim}}
\newcommand{\sd}{\mathrm{sd}}
\newcommand{\op}{\mathrm{op}}

\newcommand{\N}{\mathbb{N}}

\newcommand{\kk}{{kk}}

\newcommand{\cC}{\mathcal{C}}

\newcommand{\GAlg}{{G\Alg}}

\renewcommand{\t}[1]{{\tilde{#1}}}
\newcommand{\St}{\mathrm{St}}
\newcommand{\End}{\mathrm{End}}
\newcommand{\Hom}{\mathrm{Hom}}
\newcommand{\lev}{\mathrm{lev}}
\newcommand{\kkgr}{{{\kk_{\hspace{-0.1em}\mathrm{gr}}^G}}}
\newcommand{\Simp}{\mathrm{Simp}}
\newcommand{\IntSimp}{\mathrm{IntSimp}}
\newcommand{\card}{\mathrm{card}}

\renewcommand{\o}[1]{{\overline{#1}}}

\begin{document}

\title{Matrix stability and Morita invariance}
\author{Eugenia Ellis}
\email{eellis@fing.edu.uy}
\address{IMERL. Facultad de Ingenier\'\i a. Universidad de la Rep\'ublica. Montevideo, Uruguay.}
\author{Emanuel Rodr\'iguez Cirone}
\email{erodriguezcirone@cbc.uba.ar}
\address{Dep. Matemática -- CBC -- UBA, Buenos Aires, Argentina.}


\thanks{All authors were partially supported by grants ANII-FCE\_1\_2025\_1\_186147, \emph{Algebraic and homotopic aspects of noncommutative geometry}, PICT-2021-I-A-00710, \emph{$K$-theory, homology, and noncommutative geometry}, and UBACyT 2023 20020220300206BA, \emph{Álgebra, geometría y topología no conmutativas}. The first author was partially supported by ANII, CSIC and PEDECIBA}

\begin{abstract}
Let $G$ be a group. We prove that matrix stability for either $G$-algebras or $G$-graded algebras guarantees Morita invariance. As a consequence, bivariant algebraic $K$-theory (either $G$-equivariant or $G$-graded) is Morita invariant. In particular, we show that if $G$ is a finite group acting freely on a finite simplicial set $X$, then $\ell^X\rtimes G$ and $\ell^{X/G}$ are $\kk$-equivalent. Here,  $\ell^Y$ denotes the $\ell$-algebra of piecewise polynomial functions on $Y$ with coefficients in the ground ring $\ell$.
\end{abstract}

\maketitle
\section{Introduction}

Let $\ell$ be a commutative ring with unit. Write $\Alg$ for the category of associative and not necessarily unital $\ell$-algebras and let $\Alg^u\subseteq \Alg$ be the subcategory whose objects are the unital algebras and whose morphisms are unital algebra homomorphisms.

Two algebras $R,S\in\Alg^u$ are \textit{Morita equivalent} \cite{morita} if there  exists a tuple $(R,S,P,Q,\alpha,\beta)$ where $P$ is an $S$-$R$-bimodule, $Q$ is an $R$-$S$-bimodule, $\alpha: P\otimes_{R} Q\rightarrow S$ is an $S$-bimodule isomorphism and $\beta:Q\otimes _{S} P \rightarrow R$ is an $R$-bimodule isomorphism. The morphisms $\alpha$ and $\beta$ are moreover required to satisfy the mixed associativity relations
        \[
            \alpha(p\otimes q)\cdot p' = p\cdot \beta(q\otimes p') \qquad\text{and}\qquad
            \beta(q\otimes p)\cdot q' = q\cdot \alpha(p\otimes q')\]
        for all $p,p'\in P$ and $q,q'\in Q$.

It is well known that matrix algebras play an essential role in Morita theory. For example, any unital algebra $R$ is Morita equivalent to its algebra of $n\times n$-matrices $M_nR$. Moreover, two unital algebras $R$ and $S$ are Morita equivalent if and only if the algebras $M_\infty R$ and $M_\infty S$ are isomorphic; see, for example \cite{agr}. The notion of Morita equivalence can be generalized both to algebras with the action of a group and to graded algebras over a group. In the case of graded algebras over an \emph{abelian} group, the relation between Morita equivalences and matrix algebras we recently studied in \cite{agr}.

A functor $F:\Alg^u\rightarrow \mathcal{C}$ is called \emph{Morita invariant} if $F(R)\cong F(S)$ whenever $R$ is Morita equivalent to $S$.
Let us now consider a functor $F:\Alg \to \cC$. On one hand, $F$ is called \emph{Morita invariant} if its restriction to $\Alg^u$ is.
On the other, $F$ is called \emph{$M_{n}$-stable} if we have an isomorphism $F(A)\cong F(M_{n}A)$ induced by the upper-left corner embedding of $A\to M_nA$.
It was pointed out in \cite{cortho} that the $M_{2}$-stability of $F:\Alg \to\cC$ implies its Morita invariance. The main result of this work is a generalization of this fact to the equivariant and graded contexts.

Let $G$ be a group. We consider $\GAlg$ the category of $G$-algebras with equivariant morphisms and $\GrAlg$ the category of $G$-graded algebras with homogeneous morphisms.   
The notion of matrix stable functor has been generalized to the $G$-equivariant and $G$-graded contexts \cite{ekk}, as we proceed to recall. In both cases, we use the algebra $M_{S}$ of finitely supported matrices with coefficients in $\ell$ indexed by a set $S$. We write $E_{s,t}\in M_S$ for  the matrix whose $(s,t)$-entry is $1$ and all other entries are $0$. Note that $M_S$ is a free $\ell$-module with basis $\{E_{s,t}\}_{s,t\in S}$.

\begin{enumerate}

    \item Let $S$ be a $G$-set.
    Endow $M_S$ with the $G$-action defined by
\[g\cdot E_{s,t}= E_{g\cdot s,g\cdot t}.\]
Note that an injective morphism of $G$-sets $f : S \to T$ induces a morphism $f_*:M_S\to M_T$ by the formula $f_*(E_{s,t})=E_{f(s),f(t)}$.
If $A\in\GAlg$, we write $M_SA=M_S\otimes A$ with the diagonal action. A functor $F:\GAlg \to \cC$ is called {\emph{$G$-stable}} if $F( f_{*}:M_SA\to M_TA)$ is an isomorphism for every $G$-algebra $A$ and every injective morphism of $G$-sets $f : S \to T$ with $\operatorname{card}(T) \leq \operatorname{card}(G)$.
\item  Let $A$ be a $G$-graded algebra and let $\phi:S\to G$ be a $G$-graded set; see section \ref{sec:GGgr} for definitions and notation. Then $M_S(A)$ denotes the algebra $M_S\otimes A$ endowed with the grading
\[|E_{s,t}\otimes a|= \phi(s)|a|\phi(t)^{-1}\]
for homogeneous $a\in A$.
Note that an injective morphism of $G$-graded sets $f : S \to T$ induces a morphism $f_*:M_S(A)\to M_T(A)$ by the formula $f_*(E_{s,t}\otimes a)=E_{f(s),f(t)}\otimes a$.
A functor $F:\GrAlg \to \cC$ is called {\emph{$G$-graded stable}} if $F(f_*:M_S(A)\to M_T(A))$ is an isomorphism for every $A\in \GrAlg$ and every injective morphism of $G$-graded sets $f:S\to T$ with $\card(T)\leq \card(G)$. 
\end{enumerate}
On the other hand we have $G$-Morita equivalences in $G\Alg^{u}$, see Definition \ref{defi:GMorita}, and $G$-graded Morita equivalences in $\GrAlg^{u}$, see Definition \ref{defi:GgrMorita}.
Our main result is the following.
\begin{thm} \label{mainthm}
    Let $G$ be a group.
    \begin{enumerate}
    \item If $F:\GAlg\to\cC$ is a $G$-stable functor, then $F$ is $G$-Morita invariant.
    \item If $F:\GrAlg\to\cC$ is a $G$-graded stable functor, then $F$ is $G$-graded Morita invariant.
    \end{enumerate}
\end{thm}

An example of matrix stable fuctor is the bivariant algebraic $K$-theory (for short, $\kk$-theory) that was introduced in \cite{cortho} as an algebraic analogue to Kasparov's $KK$-theory of $C^*$-algebras. It consists of an additive category $\kk$ endowed with a functor $j:\Alg \to \kk$ that is (polynomial) homotopy invariant, matrix stable and excisive with respect to a certain family of extensions. This functor $j$ is moreover initial among the functors with the mentioned properties.
As pointed out in \cite{cortho}, the functor $j$ is Morita invariant since it is $M_2$-stable.

If $G$ is a group, $G$-equivariant and $G$-graded versions of algebraic $\kk$-theory were introduced in \cite{ekk}.
In the $G$-equivariant context, we have an additive category $\kk^G$ endowed with a universal homotopy invariant, $G$-stable and excisive functor $j^G:\GAlg\to\kk^G$; see \cite{ekk}*{Thm. 4.1.1}. For $G$-graded algebras, we have an additive category $\GrAlg$ together with a universal homotopy invariant, $G$-graded stable and excisive functor $\jgr:\GrAlg\to \kkgr$; see \cite{ekk}*{Thm. 4.2.1}.
As a consequence of Theorem \ref{mainthm} we get the following result.
\begin{coro}
    Let $G$ be a group.
    \begin{enumerate}
        \item The functor $j^G:\GAlg\to\kk^G$ is $G$-Morita invariant.
        \item The functor $\jgr:\GrAlg\to\kkgr$ is $G$-graded Morita invariant.
    \end{enumerate}
\end{coro}

Motivated by the relation between crossed product algebras and quotients in the noncommutative geometry setting developed by Alain Connes, we obtain:
\begin{thm}[cf. \cite{khal}*{Theorem 2.5.1}]\label{thm:kk}
    Let $G$ be a finite group acting freely on a finite simplicial set $X$. Suppose moreover that $\St(v)\cap \St(\overline{\St(g\cdot v)})=\emptyset$ for every $v\in X_0$ and every $g\in G$, $g\ne e$. Then the algebras $\ell^{X/G}$ and $\ell^{X}\rtimes G$ are Morita equivalent.  
\end{thm}
\noindent Here, the star of the vertex $v$, denoted by $\St(v)$, acts as a neighborhood of $v$. The condition $\St(v)\cap \St(\overline{\St(g\cdot v)})=\emptyset$ resembles a properly discontinuous action of $G$. Moreover, $\ell^Y$ denotes the algebra of piecewise polynomial functions on the simplicial set $Y$; see Appendix \ref{sec:appendix} for the precise definition.

Let us prove Theorem \ref{thm:kk} in the particular case $X=Y\times G$, with $Y$ a finite simplicial set and $g\cdot (y, h)= (y, gh)$.
We have
\begin{equation}\label{eq:khal1}\ell^{X/G}=\ell^{(Y\times G)/G}\cong\ell^{Y}\end{equation}
and
\begin{equation}\label{eq:khal2}\ell^{X} \rtimes G=  \ell^{Y\times G} \rtimes G  \cong \ell^{Y} \otimes \ell^{G} \rtimes G\overset{(\star)}{\cong}  \ell^{Y} \otimes M_{G}.\end{equation}
Here, $(\star)$ is induced by $\ell^{G} \rtimes G \cong  M_{G}$, $\chi_{g}\rtimes h \leftrightarrow E_{g,h^{-1}g}$.
Combining \eqref{eq:khal1} and \eqref{eq:khal2} we get an algebra isomorphism
\begin{equation}\label{eq:khal3}
    \ell^{X}\rtimes G\cong \ell^{X/G}\otimes M_{G}.
\end{equation}
If $G$ is a finite group of order $n$, then both $\ell^{X}\rtimes G$ and $\ell^{X/G}$ are unital algebras. Then $M_{G}=M_n$ and \eqref{eq:khal3} implies that $\ell^{X}\rtimes G$ and $\ell^{X/G}$ are Morita equivalent.

As a corollary of Theorem \ref{thm:kk} we get the following.

\begin{coro}[Corollary \ref{coro:kkeq}]
    Let $G$ be a finite group acting freely on a finite simplicial set $X$. Then $\ell^X\rtimes G$ and $\ell^{X/G}$ are $\kk$-equivalent, that is, both algebras are isomorphic in $\kk$.
\end{coro}

This paper is structured as follows. Sections \ref{sec:Gequiv} and \ref{sec:GGgr} contain the proof of Theorem \ref{mainthm} in the $G$-equivariant and $G$-graded settings, respectively. The Appendix \ref{sec:appendix} contains the proof of Theorem \ref{thm:kk} together with its corollary concerning $\kk$-theory.

\section{G-equivariant case}\label{sec:Gequiv}

\subsection{\texorpdfstring{$G$}{G}-Algebras and equivariant morphisms}
A \emph{$G$-algebra} is an algebra with an action of $G$ by algebra automorphisms. An \emph{equivariant morphism} between $G$-algebras is a $G$-equivariant algebra homomorphism. We write $G\Alg$ for the category of $G$-algebras with equivariant morphisms. A $G$-algebra $A$ is \emph{unital} if $A$ is a unital algebra and $g\cdot 1=1$ for all $g\in G$. We write $\GAlg^u\subseteq \GAlg$ for the subcategory whose objects are unital $G$-algebras and whose morphisms are unital morphisms. 

\subsection{Matrix stability and \texorpdfstring{$G$-}{G-}stability}
Let $S$ be a set and pick $s_0\in S$. Recall that if $A\in\GAlg$, then $M_SA=M_S\otimes A$ denotes the algebra of finitely supported matrices with coefficients in $A$ indexed by $S$. We consider $M_SA$ as a $G$-algebra with the $G$-action given by $g\cdot (E_{s,t}\otimes a)= E_{s,t}\otimes g\cdot a$. The inclusion $\iota_{s_0}:A \to M_S A$, $\iota_{s_0}(a)=E_{s_0,s_0}\otimes a$, is a non-unital equivariant morphism. A functor $F: \GAlg \to \cC$ is called \emph{$M_{S}$-stable} if $F(\iota_{s_0})$ is an isomorphism for any $A\in \GAlg$. It is well known that this does not depend upon the choice of $s_0$; see \cite{friendly}*{Section 2.2}. 

If $S$ is a $G$-set, the $G$-action on $M_SA$ is given by $g\cdot (E_{s,t}\otimes a)=E_{g\cdot s, g\cdot t}\otimes g\cdot a$. To avoid confusions, we write $|S|$ for the underlying set of $S$, so that $G$ acts on $M_{|S|}A$ only on the coefficients while it acts on $M_SA$ both on the coefficients and the indices. Note that the formula $\iota_{s_0}(a)=E_{s_0,s_0}\otimes a$ does not define an equivariant morphism $A\to M_SA$ unless $s_0$ is fixed by $G$.

\begin{lem}[cf. \cite{ekk}*{Rem. 3.1.5}]\label{lem:CDAequiv}
    Let $S$ be a $G$-set. Then the formula
    \[E_{s,t}\otimes E_{g,h}\mapsto E_{g\cdot s,h\cdot t}\otimes E_{g,h}\]
    defines an isomorphism of $G$-algebras $M_{|S|}\otimes M_G\cong M_S\otimes M_G$.
\end{lem}
\begin{proof}
    It is a straightforward verification.
\end{proof}

\begin{defi}
Let $\cC$ be a category. A functor $F:\GAlg \to \cC$ is {\emph{$G$-stable}} if for every $G$-algebra $A$ and every injective equivariant function $f : S \to T$ with $\card(T) \leq \card(G)$, the morphism 
$F( f_{*}: M_{S} A \to M_{T}A)$ is an isomorphism in $\cC$.
\end{defi}

\begin{rem}
    There is a slight difference between the definition above and the one given in \cite{ekk}*{Section 3} because we are not requiring a $G$-stable functor to be $M_\N$-stable unless $G$ is infinite.
\end{rem}

Write $G_{+}=G\sqcup\{*\}$ for the disjoint union of $G$ with the trivial one point $G$-set. 
The inclusions $\{*\} \subset G_{+}$ and $G\subset G_{+}$ induce the following zig-zag
of equivariant morphisms, for any G-algebra $A$:
$$
A\xrightarrow[]{\iota_{A}} M_{G_{+}}A \xleftarrow[]{\iota'_{A}}  M_{G}A $$
If $F:\GAlg\to \cC$ is $G$-stable, then $F(\iota_A)$ and $F(\iota'_A)$ are isomorphisms. A partial converse holds by the following lemma.

\begin{lem}\label{lemeq}
    Let $F:\GAlg\to \cC$ be an $M_{|G|}$-stable functor that inverts the morphisms $\iota_A$ and $\iota'_A$ for every $G$-algebra $A$. Then $F$ is $G$-stable.
\end{lem}
\begin{proof}
    Let $f:S\to T$ be an injective equivariant function $f : S \to T$ with $\operatorname{card}(T) \leq \operatorname{card}(G)$ and let $A$ be any algebra. We must show that the left vertical morphism in the diagram below induces an isomorphism upon applying $F$.
    \[
\xymatrix{
    M_{S}\otimes A \ar[r]^-{\id \otimes \iota }\ar[d]_{f_{*}\otimes \id} & M_{S}\otimes M_{G_{+}}\otimes A \ar[d]_{f_{*}\otimes \id} & M_{S} \otimes M_G \otimes A \ar[l]_-{\id\otimes\iota'} \cong  M_{|S|} \otimes M_G \otimes A\ar[d]_{f_{*}\otimes \id} \\
    M_T\otimes A\ar[r]_-{\id\otimes \iota} & M_T\otimes M_{G_+}\otimes A & M_{T} \otimes M_G \otimes A  \cong  M_{|T|} \otimes M_G \otimes A\ar[l]^-{\id\otimes\iota'}
    }   
    \]
    The horizontal morphisms induce isomorphisms upon applying $F$ by hypothesis. The isomorphisms on the right are induced by Lemma \ref{lem:CDAequiv}. The right vertical morphism induces an isomorphism upon applying $F$ since $F$ is $M_{|G|}$-stable. Indeed, fix $s_0\in S$, let $t_0=f(s_0)$ and consider the following triangle:
    \[\xymatrix{M_{|S|} \otimes M_G \otimes A\ar[rr]^-{f_*\otimes \id} & & M_{|T|} \otimes M_G \otimes A \\
    & M_G\otimes A\ar[ul]^-{\iota_{s_0}}\ar[ur]_-{\iota_{t_0}} &}\]
    We know that $F(\iota_{s_0})$ and $F(\iota_{t_0})$ are isomorphisms by $M_{|G|}$-stability and it follows that so is $F(f_*\otimes\id)$.
\end{proof}

\begin{rem} Let $F:\GAlg\to \cC$ be an $M_{\N}$-stable functor that inverts the morphisms $\iota_A$ and $\iota'_A$ for every $G$-algebra $A$. The arguments in the proof of Lemma \ref{lemeq} together with \cite{ekk}*{Example 3.1.3}, \cite{ekk}*{Remark 3.1.5} and \cite{ekk}*{Remark 3.1.6} imply that $F$ is $G$-stable in the sense of \cite{ekk}*{Section 3.1}. In particular, the original definition of a $G$-stable functor given in \cite{ekk}*{Section 3.1} is equivalent to the one given above. See also \cite{arcort}*{Corollary 6.9}.
\end{rem}

\begin{lem}\label{lem:Gcorner}
    Let $R\in\GAlg$ be a unital $G$-algebra, let $m\in\N$ and let $u, v\in R^{m\times 1}$ be such that $u^tv=1_R$. Let $\iota_m:M_G R\to M_G M_mR$ be induced by the upper-left corner inclusion and let $\iota_{u,v}:M_GR\to M_GM_mR$ be the morphism given by
    \[\iota_{u,v}(E_{g,h} \otimes r)=E_{g,h} \otimes (g\cdot v)r(h\cdot u)^t.\]
    If $F:\GAlg\to \cC$ is an $M_2$-stable functor, then $F(\iota_{u,v})=F(\iota_m)$. In particular, $F(\iota_{u,v})$ is an isomorphism. 
\end{lem}
\begin{proof}
    Let $e_1\in R^{m\times 1}$ be the first canonical basis vector, and consider the idempotents $q:=e_1e_1^t$ and $p:=vu^t$ in $M_mR$. Then $q$ and $p$ are Murray-von Neumann equivalent. Indeed, putting $a=e_1u^t$ and $b=ve_1^t$ we have $ab=q$ and $ba=p$. Define $V\in M_{2m}R$ as the block matrix:
    \[V=\begin{pmatrix}
        a & 1-q \\ 1-p & b
    \end{pmatrix}\]
    Using the identities $qa=a=ap$ and $pb=b=bq
    $, it is easily verified that $V$ is invertible with inverse
    \[V^{-1}=\begin{pmatrix}
        b & 1-p \\ 1-q & a
    \end{pmatrix}.\]
    We claim that there is a commutative diagram of $G$-algebras as follows, where $\iota_2$ is the inclusion into the upper-left corner.
    \begin{equation}\label{eq:cornersG}\begin{gathered}\xymatrix{
    M_GR\ar[r]^-{\iota_m}\ar@/_1pc/[dr]_-{\iota_{u,v}} & M_GM_mR\ar[r]^{\iota_2} & M_GM_{2m}R \\
    & M_GM_mR\ar[r]^-{\iota_2} & M_GM_{2m}R\ar[u]_-{\Phi^W}
    }\end{gathered}\end{equation}
    Here, $\Phi^W$ is conjugation by the matrix $W=\sum_{g\in G}E_{g,g}\otimes(g\cdot V)$. Note that if $G$ is infinite, then $W$ has infinitely many non-zero entries and thus $W\not\in M_GM_{2m}R$. This is not a problem, though, because $W$ lies in the bigger \emph{cone algebra} \cite{friendly}*{Section 2.3}, which is a unital $G$-algebra that includes $M_GM_{2m}R$ as a two-sided ideal.
    By a direct computation, proving the commutativity of \eqref{eq:cornersG} amounts to verifying the equality
    \[(g\cdot V)\begin{pmatrix}
        (g\cdot v)r(h\cdot u)^t & 0 \\ 0 & 0 
    \end{pmatrix}(h\cdot V^{-1})=\begin{pmatrix}
        e_1 r e_1^t & 0 \\ 0 & 0 
    \end{pmatrix}\]
    for $r\in R$ and $g,h\in G$.
    Upon applying $F$ to \eqref{eq:cornersG} we get that 
    \[F(\Phi^W)\circ F(\iota_2)\circ F(\iota_{u,v})=F(\iota_2)\circ F(\iota).\]
    As $F(\Phi^W)$ is the identity by \cite{er}*{Lemma 2.12} and $F(\iota_2)$ is invertible by $M_2$-stability, it follows that $F(\iota_{u,v})=F(\iota)$.
\end{proof}

\subsection{\texorpdfstring{$G$}{G}-Morita equivalence}

\begin{defi}\label{defi:GMorita}
    Two unital $G$-algebras $R$ and $S$ are \emph{$G$-Morita equivalent} if there exists a tuple $(R,S,P,Q,\alpha,\beta)$ satisfying the following conditions:
    \begin{itemize}
        \item The sets $P$ and $Q$ are $G$-modules equipped with an $S$-$R$-bimodule structure and an $R$-$S$-bimodule structure, respectively. The actions of $R$ and $S$ on $P$ and $Q$ are moreover compatible with the $G$-actions. That is, we have
        \[
            g\cdot (s\cdot p\cdot r) = (g\cdot s)\cdot (g\cdot p)\cdot (g\cdot r)
        \]
        \[
            g\cdot (r\cdot q\cdot s) = (g\cdot r)\cdot (g\cdot q)\cdot (g\cdot s)
        \]
        for all $r\in R$, $s\in S$, $p\in P$, $q\in Q$, and $g\in G$.

        \item The modules $P\otimes_R Q$ and $Q\otimes_S P$ are equipped with the diagonal $G$-actions:
        \[g\cdot (p\otimes q)=(g\cdot p)\otimes (g\cdot q)\]
        \[g\cdot (q\otimes p)=(g \cdot q)\otimes (g\cdot p)\]
        The maps $\alpha: P\otimes_R Q\to S$ and $\beta: Q\otimes_S P\to R$ are $G$-equivariant bimodule isomorphisms satisfying the mixed associativity relations
        \[
            \alpha(p\otimes q)\cdot p' = p\cdot \beta(q\otimes p')
        \]
        \[
            \beta(q\otimes p)\cdot q' = q\cdot \alpha(p\otimes q')
        \]
        for all $p,p'\in P$ and $q,q'\in Q$.
    \end{itemize}
    We call such a tuple $(R,S,P,Q,\alpha,\beta)$ a \emph{$G$-Morita equivalence}.
\end{defi}

\begin{rem}
    In the particular case where $G=1$, Definition \ref{defi:GMorita} recovers the usual notion of Morita equivalence between unital algebras.
\end{rem}
\begin{rem} If $(R,S,P,Q,\alpha,\beta)$ is a $G$-Morita equivalence, we can consider $P\otimes_R Q$ and $ Q\otimes_S P$ as $G$-algebras with the following products:       
\[
(p\otimes q)(p'\otimes q')=p\cdot \beta(q\otimes p')\otimes q'
\]
\[
(q\otimes p)(q'\otimes p')=q \cdot \alpha(p\otimes q')\otimes p'
\]
With this structure, $\alpha$ and $\beta$ are isomorphisms of $G$-algebras. 
\end{rem}

\begin{defi}
    Let $\cC$ be a category. A functor $F:\GAlg^u\to \cC$ is \emph{$G$-Morita invariant} if $F(R)\cong F(S)$ whenever $R$ and $S$ are $G$-Morita equivalent unital $G$-algebras. A functor $F:\GAlg\to \cC$ is called \emph{$G$-Morita invariant} if its restriction to $\GAlg^u$ is.
\end{defi}

\begin{con}[The corner $G$-algebra]\label{con:GMorita}
Let $(R,S,P,Q,\alpha,\beta)$ be a $G$-Morita equivalence. Since $\alpha: P \otimes_R Q \to S$ is an isomorphism, we may choose $n \in \mathbb{N}$ and elements $p_1, \dots, p_n \in P$ and $q_1, \dots, q_n \in Q$ such that $\alpha(\sum_{i=1}^n p_i \otimes q_i) = 1$. 
For each $g \in G$, define $U_g, V_g \in M_nR$ by
\begin{align*}
    U_g &:= \sum_{i,j} E_{i,j} \otimes \beta(q_i \otimes (g \cdot p_j)), \\
    V_g &:= \sum_{i,j} E_{i,j} \otimes \beta((g \cdot q_i) \otimes p_j).
\end{align*}
Set also $d := U_1 = V_1$. Then the following assertions hold.
\begin{enumerate}[label=(\alph*)]
    \item The matrix $d$ is an idempotent.
    \item We have $dU_g = U_g$, $V_gd = V_g$, $U_g V_g = d$ and $V_g U_g = g \cdot d$ for all $g \in G$.
    \item We have $U_g(g \cdot U_h) = U_{gh}$ and $(g \cdot V_h)V_g = V_{gh}$ for all $g,h \in G$.
    \item Let $M^d_nR$ denote the corner algebra $d(M_nR)d$. The assignment $\theta_g(T) = U_g(g \cdot T)V_g$ defines a $G$-action $\theta$ on the unital algebra $M_n^dR$. We will always consider the corner algebra $M^d_nR$ as a $G$-algebra with the $G$-action $\theta$.
    \item There is an equivariant morphism $\xi: S \to M^d_nR$ defined by
    \[ \xi(s) = \sum_{i,j} E_{i,j} \otimes \beta((q_i \cdot s) \otimes p_j). \]
    \item There is an equivariant morphism $\zeta:M_GM^d_nR\to M_GM_nR$ defined by
    \[\zeta(E_{h,k}\otimes T)=E_{h,k}\otimes V_hTU_k.\]
\end{enumerate}
The verification of all the properties above is straightforward from the definitions. 
\end{con}

\begin{prop}\label{prop:iso_tensorMG}
    Let $F:\GAlg\to \cC$ be an $M_2$-stable functor and let  $R,S\in\GAlg$ be unital $G$-algebras. If $R$ and $S$ are $G$-Morita equivalent, then \[F(M_GR)\cong F(M_GS).\]
\end{prop}
\begin{proof}
    Let $(R,S,P,Q,\alpha,\beta)$ be a $G$-Morita equivalence and define $d$, $\theta$, $\xi$, and $\zeta$ as in Construction \ref{con:GMorita}. 
    Let $\Phi:F(M_GS)\to F(M_GR)$ be the following composite, where the rightmost arrow is the inverse of the morphism induced by the upper-left-corner inclusion $\iota_n$.
\[\xymatrix@C=2em{F(M_GS)\ar@/_2pc/[rrr]^-{\Phi}\ar[r]^-{\xi_*} & F(M_GM^d_nR)\ar[r]^-{\zeta_*} & F(M_GM_nR)\ar[r]_-{\cong}^-{(\iota_n)_*^{-1}} & F(M_GR)}\]
We will show that $\Phi$ is an isomorphism. By symmetry, since $(S,R,Q,P,\beta,\alpha)$ is also a $G$-Morita equivalence, we may find $\t{n}\in\N$, $\t{p}_1,\dots, \t{p}_\t{n}\in P$ and $\t{q}_1,\dots, \t{q}_\t{n}\in Q$ such that $\beta(\sum_{i=1}^{\t{n}}\t{q}_i\otimes\t{p}_i)=1$. By Construction \ref{con:GMorita} we get an idempotent $\t{d}\in M_\t{n}S$, a corner $G$-algebra $M^\t{d}_\t{n}S$ and morphisms of $G$-algebras $\t{\xi}:R\to M^\t{d}_\t{n}S$ and $\t{\zeta}:M_GM^\t{d}_\t{n}S\to M_GM_\t{n}S$. Let $\Psi:F(M_GR)\to F(M_GS)$ be the following composite.
\[\xymatrix@C=2em{F(M_GR)\ar@/_2pc/[rrr]^-{\Psi}\ar[r]^-{\t{\xi}_*} & F(M_GM^\t{d}_\t{n}S)\ar[r]^-{\t{\zeta}_*} & F(M_GM_\t{n}S)\ar[r]_-{\cong}^-{(\iota_\t{n})_*^{-1}} & F(M_GS)}\]
We claim that $\Psi$ is the inverse of $\Phi$. Indeed, consider the following commutative diagram:
\[\xymatrix@C=5em{F(M_GS)\ar@/_1pc/[dr]_-{\Phi}\ar[r]^-{(\zeta\circ \xi)_*} & F(M_GM_nR)\ar[r]^-{(\t{\zeta}\circ \t{\xi})_*} & F(M_GM_nM_\t{n}S)\\
& F(M_GR)\ar@/_1pc/[dr]_-{\Psi}\ar[r]^-{(\t{\zeta}\circ\t{\xi})_*}\ar[u]^-{\cong}_-{(\iota_n)_*} & F(M_GM_\t{n}S)\ar[u]^-{\cong}_-{(\iota_n)_*} \\
& & F(M_GS)\ar[u]^-{\cong}_-{(\iota_\t{n})_*}
}\]
To prove that $\Psi\circ\Phi=\id$ it suffices to show that the composite of the morphisms in the top row equals the morphism induced by the upper-left-corner inclusion. Let $H:M_GS\to M_GM_nM_\t{n}S$ be the equivariant morphism inducing the composite of the morphisms in the top row. A direct computation shows that $H$ is given by the formula
\[H(E_{g,h} \otimes s)=  E_{g,h} \otimes \sum_{i,j=1}^n\sum_{k,l=1}^\t{n}E_{i,j}\otimes E_{k,l}\otimes (g\cdot \alpha(\t{p}_k\otimes q_i))s(h\cdot \alpha(p_j\otimes \t{q}_l)).\]
Thus, upon applying $F$, the morphism $H:M_GS\to M_GM_nM_\t{n}S$  induces the same morphism as the upper-left-corner inclusion by Lemma \ref{lem:Gcorner} (take $m=n\t{n}$ and identify $M_m\cong M_n\otimes M_\t{n}$). This proves that $\Psi\circ\Phi=\id$. The fact that $\Phi\circ\Psi=\id$ follows by symmetry.
\end{proof}

\begin{coro}[\cite{cortho}*{Remark 5.1.3}]\label{cor:gid}
    Let $F:\Alg\to\cC$ be a functor. If $F$ is $M_2$-stable, then $F$ is Morita invariant.
\end{coro}
\begin{proof}
    Put $G=1$ in Proposition \ref{prop:iso_tensorMG}.
\end{proof}

\begin{coro}
    Let $F:\GAlg\to\cC$ be a functor. If $F$ is $G$-stable, then $F$ is $G$-Morita invariant.
\end{coro}
\begin{proof}
    Let $R$ and $S$ be unital $G$-algebras. If $R$ and $S$ are $G$-Morita equivalent, we have
    \[F(R)\cong F(M_GR)\cong F(M_GS)\cong F(S).\]
    Indeed, the isomorphism in the middle holds by Proposition \ref{prop:iso_tensorMG} and the other two isomorphisms hold by $G$-stability.
\end{proof}

\section{\texorpdfstring{$G$-Graded}{G-Graded} case}\label{sec:GGgr}

\subsection{\texorpdfstring{$G$}{G}-Graded algebras and homogeneous morphisms}
A \emph{$G$-graded algebra} is an algebra $A\in\Alg$ with a family of submodules 
$\{A_g\}_{g\in G}$ such that
$$
A = \bigoplus_{g\in G} A_{g}, \qquad A_{g} A_{h} \subseteq A_{gh} \qquad (g, h \in G). 
$$
We write $|a|=g$  if $a \in A_{g}$ and call $a$ \emph{homogeneous of degree} $g$. An {\emph{homogeneous morphism}} between $G$-graded algebras is an algebra homomorphism $f:A\rightarrow B$ such that $f(A_{g}) \subseteq B_{g}$, $\forall g\in G$. We write $\GrAlg$ for the category of $G$-graded algebras with homogeneous morphisms. A $G$-graded algebra $A$ is \emph{unital} if $A$ is a unital algebra and $|1|=e$. We write $\GrAlg^u\subseteq \GrAlg$ for the subcategory whose objects are unital $G$-graded algebras and whose morphisms are unital homogeneous morphisms.

\subsection{Matrix stability}\label{sec:grmatrixst}
Let $S$ be a set, pick $s_0\in S$ and let $A\in\GrAlg$. Recall that $M_S$ denotes the algebra of finitely supported matrices with coefficients in $\ell$ indexed by $S$. We write $M_SA$ for the algebra $M_S\otimes A$ with the grading given by $|E_{s,t}\otimes a|= |a|$ for homogeneous $a$. The inclusion $\iota_{s_0}:A \to M_S A$, $\iota_{s_0}(a)=E_{s_0,s_0}\otimes a$, is a non-unital homogeneous morphism. A functor $F: \GrAlg \to \cC$ is called \emph{$M_{S}$-stable} if $F(\iota_{s_0})$ is an isomorphism for any $A\in \GrAlg$. It is well known that this does not depend upon the choice of $s_0$; see \cite{friendly}*{Section 2.2}. 

\begin{lem}[cf. \cite{friendly}*{Proposition 2.2.6}]\label{lem:M2conjugategr}
    Let $F:\GrAlg\to\cC$ be an $M_2$-stable functor and let $A$ be a $G$-graded algebra. Suppose that $A$ is a subalgebra of a not-necessarily-graded unital algebra $B$. Let $V\in B$ be an invertible element such that $VA\subseteq A$, $AV^{-1}\subseteq A$ and $|VaV^{-1}|=|a|$ for every homogeneous $a\in A$. Then
    \[\Phi^V(a)=VaV^{-1}\]
    defines an homogeneous morphism $\Phi^V:A\to A$ and $F(\Phi^V)=\id_{F(A)}$.    
\end{lem}
\begin{proof}
    The proof of \cite{friendly}*{Proposition 2.2.6} can be carried on almost verbatim in this $G$-graded setting.
\end{proof}

\subsection{\texorpdfstring{$G$-Graded stability}{G-Graded stability}} Let $A$ be a $G$-graded algebra. Recall that the algebra $M_G(A)=M_G\otimes A$ is a $G$-graded algebra with the grading defined by $|E_{g,h}\otimes a|=g|a|h^{-1}$ for homogeneous $a$. This construction can be generalized as follows.

\begin{defi}
    A \emph{$G$-graded set} consists of a set $S$ endowed a function $\phi:S\to G$. Let $\phi:S\to G$ and $\psi:T\to G$ be $G$-graded sets. A \emph{morphism} from $S$ to $T$ is a function $f:S\to T$ such that $\psi\circ f=\phi$.
\end{defi}

\begin{defi}
    Let $\phi: S \to G$ be a $G$-graded set and let $A$ be a $G$-graded algebra. Then $M_S\otimes A$ is a $G$-graded algebra with the grading defined by
    \[|E_{s,t}\otimes a|=\phi(s)|a|\phi(t)^{-1}.\]
    We write $M_S(A)$ for the $G$-graded algebra defined this way. 
    When considering $G$ as a $G$-graded set, we use the grading induced by $\phi=\id_G$.
\end{defi}

\begin{rem}\label{rem:grmatrices}
    Let $\phi:S\to G$ and $\psi:T\to G$ be $G$-graded sets and let $f:S\to T$ be an injective morphism of $G$-graded sets. For any $A\in\GrAlg$ we have an induced homogeneous morphism $f_*:M_S(A)\to M_T(A)$, $f_*(E_{s,t}\otimes a)=E_{f(s),f(t)}\otimes a$.
\end{rem}

\begin{rem}
Any $G$-graded set $S$ has an underlying set denoted by $|S|$. Thus, the grading in $M_{|S|}A$ is given by $|E_{s,t}\otimes a|=|a|$, as defined in section \ref{sec:grmatrixst}.
\end{rem}

\begin{lem}[cf. \cite{ekk}*{Prop. 3.2.2}]\label{lem:CDAgraduada}
    Let $\phi:S\to G$ be a $G$-graded set and let $A$ be a $G$-graded algebra. Then the formula
    \[E_{g,h}\otimes E_{s,t}\otimes a\mapsto E_{s,t}\otimes E_{g\phi(s), h\phi(t)}\otimes a\]
    defines an isomorphism of $G$-graded algebras $M_G(M_S(A))\to M_{|S|}M_G(A)$.
\end{lem}
\begin{proof}
    It is a straightforward verification.
\end{proof}

\begin{defi}
    Let $\cC$ be a category. A functor $F:\GrAlg\to \cC$ is  \emph{$G$-graded stable} if for every $A\in\GrAlg$ and every injective morphism of $G$-graded sets $f:S\to T$ with $\card(T)\leq \card(G)$, the morphism $F(f_*:M_S(A)\to M_T(A))$ is an isomorphism.
\end{defi}

\begin{exa}
    Put $S=\{e\}$, $T=G$, and let $\phi:S\to G$, $\psi:T\to G$, $f:S\to T$ be the inclusions. It is straightforward to verify that $f_*:M_S(A)\to M_T(A)$ of Remark \ref{rem:grmatrices} is identified with the inclusion $\iota_e:A\to M_G(A)$, $\iota_e(a)=E_{e,e}\otimes a$. In particular, any $G$-graded stable functor $F:\GrAlg\to\cC$ sends $\iota_e$ to an isomorphism. A partial converse is given by the following result.
\end{exa}

\begin{lem}\label{lemeqgr}
    Let $F:\GrAlg\to \cC$ be an $M_{|G|}$-stable functor that inverts the morphism $\iota_e:A\to M_G(A)$, $\iota_e(a)=E_{e,e}\otimes a$, for every $G$-graded algebra $A$. Then $F$ is $G$-graded stable.
\end{lem}
\begin{proof}
    Let $A\in\GrAlg$ and let $f:S\to T$ be an injective function as in Remark \ref{rem:grmatrices}. Fix $s_0\in S$ and let $t_0=f(s_0)$. We have a commutative diagram of $G$-graded algebras as follows, where the morphisms marked with $\sim$ become isomorphisms upon applying $F$.
    \[\xymatrix@C=4em{M_S(A)\ar[r]^-{\iota_e}_-{\sim}\ar[d]_-{f_*} & M_G(M_S(A))\ar[d]_-{f_*}\ar[r]^-{\text{Lem. \ref{lem:CDAgraduada}}}_-{\cong} & M_{|S|}M_G(A)\ar[d]_-{f_*} & M_G(A)\ar[l]_-{\iota_{s_0}}^-{\sim}\ar@/^1pc/[dl]_-{\iota_{t_0}}^-{\sim}\\
    M_T(A)\ar[r]^-{\iota_e}_-{\sim} & M_G(M_T(A))\ar[r]^-{\text{Lem. \ref{lem:CDAgraduada}}}_-{\cong} & M_{|T|}M_G(A) & }\]
    Note that both $\iota_{s_0}$ and $\iota_{t_0}$ induce isomorphisms upon applying $F$ by $M_{|G|}$-stability. It follows that $F(f_*:M_S(A)\to M_T(A))$ is an isomorphism as well.
\end{proof}

\begin{lem}\label{lem:Gcornergr}
    Let $R\in\GrAlg$ be a unital $G$-graded algebra, let $n\in\N$ and let $u, v\in R^{n\times 1}$ be such that $u^tv=1_R$. Suppose moreover that $u_i$ and $v_i$ are homogeneous elements with $|v_i|=|u_i|^{-1}$. Let $\iota_n:M_G(R)\to M_nM_G(R)$ be the upper-left corner inclusion and let $\iota_{u,v}:M_G(R)\to M_nM_G(R)$ be the graded homomorphism defined by
    \[\iota_{u,v}(E_{g,h}\otimes r)=\sum_{i,j}E_{i,j}\otimes E_{g|u_i|, h|u_j|}\otimes v_iru_j.\]
    If $F:\GrAlg\to \cC$ is an $M_2$-stable functor, then $F(\iota_{u,v})=F(\iota_n)$. In particular, $F(\iota_{u,v})$ is an isomorphism. 
\end{lem}
\begin{proof}
    It is straightforward to verify that $\iota_{u,v}$ indeed defines a morphism in $\GrAlg$.
    
    Let $\Gamma_G(R)$ be the (ungraded) algebra of those $G\times G$ matrices with coefficients in $R$ that have finitely many nonzero entries in each row and column. 
    Then $\Gamma_G(R)$ is a unital algebra and $M_G(R)\subset \Gamma_G(R)$ is a two sided ideal. Consider the following elements of $\Gamma_G(R)$:
    \begin{align*}
        p(i,j)&= \sum_{g\in G} E_{g|u_i|, g|u_j|}\otimes v_iu_j && \text{($1\leq i, j\leq n$)} \\
        a(1,j)&= \sum_{g\in G} E_{g,g|u_j|}\otimes u_j &&\text{($1\leq j\leq n$)} \\
        b(i,1)&=\sum_{g\in G} E_{g|u_i|,g}\otimes v_i && \text{($1\leq i\leq n$)}
    \end{align*}
    They indeed belong to $\Gamma_G(R)$ since all of them have exactly one nonzero entry in each row and column. An easy computation shows that if $x\in M_G(R)$ is an homogeneous element and $y\in \Gamma_G(R)$ is one of the elements  $p(i,j)$, $a(1,j)$ or $b(i,1)$, then both $xy,yx\in M_G(R)$ are homogeneous with $|xy|=|yx|=|x|$. For example, if $r\in R$ is homogeneous we have
    \[
        |a(1,j)(E_{h,k}\otimes r)|=|E_{h|u_j|^{-1},k}\otimes u_jr|=h|u_j|^{-1}|u_j||r|k^{-1}=h|r|k^{-1}=|E_{h,k}\otimes r|. 
    \]
    This implies that if $x\in M_G(R)$ is homogeneous and  $y=a(1,j)$, then $yx\in M_G(R)$ is homogeneous and $|yx|=|x|$. It is also straightforward to show that the following identities hold:
    \begin{align*}
        p(i,k)p(k,j)&=p(i,j) \\
        a(1,k)p(k,j)&=a(1,j) \\
        p(i,k)b(k,1)&=b(i,1) \\
        a(1,k)b(k,1)&=1 \\
        b(i,1)a(1,j)&= p(i,j)
    \end{align*}
    Now define $q, p,a,b\in M_{n}\Gamma_G(R)$ as follows:
    \begin{align*}
        q &= E_{1,1}\otimes 1 \\
        p &= \sum_{i,j} E_{i,j} \otimes p(i,j) \\
        a &= \sum_j E_{1,j}\otimes a(1,j) \\
        p &= \sum_i E_{i,1} \otimes b(i,1)
    \end{align*}
    Using the identities above, the following equalities are easily verified:
    \begin{equation}\label{eq:lemidentities}q^2=q,\quad p^2=p,\quad qa=a=ap,\quad pb=b=bq,\quad ab=q,\quad ba=p.\end{equation}
    Moreover, it also follows from the above that if $x\in M_nM_G(R)$ is homogeneous and $y$ is either $q$, $p$, $a$ or $b$, then $xy,yx\in M_nM_G(R)$ are homogeneous as well and $|xy|=|yx|=|x|$. Define $V\in M_2M_n\Gamma_G(R)$ by
    \[V=\begin{pmatrix}
        a & 1-q \\ 1-p & b
    \end{pmatrix}\]
    It follows from \eqref{eq:lemidentities} that $V$ is invertible with inverse
    \[V^{-1}=\begin{pmatrix}
        b & 1-p \\ 1-q & a
    \end{pmatrix}\]
    Moreover, if $x\in M_2M_n M_G(R)$ is homogeneous then $Vx, xV^{-1}\in M_2M_nM_G(R)$ are homogeneous as well and $|Vx|=|xV^{-1}|=|x|$.
    A short computation shows that the following diagram in $\GrAlg$ commutes, where $\Phi^V(x)=VxV^{-1}$.
    \begin{equation}\label{eq:cornersGgr}\begin{gathered}\xymatrix{
    M_G(R)\ar[r]^-{\iota_n}\ar@/_1pc/[dr]_-{\iota_{u,v}} & M_nM_G(R)\ar[r]^{\iota_2} & M_2M_nM_G(R) \\
    & M_nM_G(R)\ar[r]^-{\iota_2} & M_2M_nM_G(R)\ar[u]_-{\Phi^V}
    }\end{gathered}\end{equation}
    Upon applying $F$ to \eqref{eq:cornersGgr} we get that 
    \[F(\Phi^V)\circ F(\iota_2)\circ F(\iota_{u,v})=F(\iota_2)\circ F(\iota).\]
    As $F(\Phi^V)$ is the identity by Lemma \ref{lem:M2conjugategr} and $F(\iota_2)$ is invertible by $M_2$-stability, it follows that $F(\iota_{u,v})=F(\iota)$.
\end{proof}

\subsection{\texorpdfstring{$G$-Graded Morita}{G-Graded Morita} equivalence}

\begin{defi}\label{defi:GgrMorita}
    Two unital $G$-graded algebras $R$ and $S$ are \emph{$G$-graded Morita equivalent} if there exists a tuple $(R,S,P,Q,\alpha,\beta)$ satisfying the following conditions:
    \begin{itemize}
        \item The sets $P$ and $Q$ are $G$-graded modules equipped with an $S$-$R$-bimodule structure and an $R$-$S$-bimodule structure, respectively. The actions of $R$ and $S$ on $P$ and $Q$ are moreover compatible with the $G$-gradings
        \[
            |s\cdot p\cdot r| = |s||p||r|
        \]
        \[
           |r\cdot q\cdot s| = |r||q||s|
        \]
        for all $r\in R$, $s\in S$, $p\in P$, $q\in Q$.

\item The modules $P\otimes_R Q$ and $Q\otimes_{S} P$ are $G$-graded with
\[
|p\otimes q|=|p||q|\]\[|q\otimes p|=|q||p| 
\]
for homogeneous $p$ and $q$.
The maps $\alpha: P\otimes_R Q\to S$ and $\beta: Q\otimes_S P\to R$ are $G$-graded bimodule isomorphisms satisfying the mixed associativity relations
        \[
            \alpha(p\otimes q)\cdot p' = p\cdot \beta(q\otimes p')
        \]
        \[
            \beta(q\otimes p)\cdot q' = q\cdot \alpha(p\otimes q')
        \]
        for all $p,p'\in P$ and $q,q'\in Q$.
       \end{itemize}
    We call such a tuple $(R,S,P,Q,\alpha,\beta)$ a \emph{$G$-graded Morita equivalence}.
\end{defi}

\begin{rem} If $(R,S,P,Q,\alpha,\beta)$ is a $G$-graded Morita equivalence, we can consider $P\otimes_R Q$ and $ Q\otimes_S P$ as $G$-graded algebras with the following products:       
\[
(p\otimes q)(p'\otimes q')=p\cdot \beta(q\otimes p')\otimes q'
\]
\[
(q\otimes p)(q'\otimes p')=q \cdot \alpha(p\otimes q')\otimes p'
\]
With this structure, $\alpha$ and $\beta$ are isomorphisms of $G$-graded algebras. 
\end{rem}

\begin{defi}
    Let $\cC$ be a category. A functor $F:\GrAlg^u\to \cC$ is \emph{$G$-graded Morita invariant} if $F(R)\cong F(S)$ whenever $R$ and $S$ are $G$-graded Morita equivalent unital $G$-graded algebras. A functor $F:\GrAlg\to \cC$ is called \emph{$G$-graded Morita invariant} if its restriction to $\GrAlg^u$ is.
\end{defi}

\begin{prop}\label{prop:iso_tensorMGgr}
    Let $F:\GrAlg\to \cC$ be an $M_2$-stable functor and let  $R$ and $S$ be unital $G$-graded algebras. If $R$ and $S$ are $G$-graded Morita equivalent, then \[F(M_G(R))\cong F(M_G(S)).\]
\end{prop}
\begin{proof}
    Let $(R,S,P,Q,\alpha,\beta)$ be a $G$-graded Morita equivalence. Since $\alpha:P\otimes_R Q\to S$ is an isomorphism, there exist $n\in\N$, $p_1,\dots, p_n\in P$ and $q_1,\dots, q_n\in Q$ such that $\alpha(\sum_{i=1}^n p_i\otimes q_i)=1_S$.
    
    We may assume that both the $p_i$ and the $q_i$ are homogeneous with $|q_i|=|p_i|^{-1}$. Indeed, by decomposing each $p_i$ and $q_i$ into their homogeneous components, $p_i=\sum_g p_i^g$ and $q_i=\sum_gq_i^g$, we have:
    \[1_S=\alpha\left(\sum_{i=1}^n p_i\otimes q_i\right)=\sum_{k\in G}\alpha\left(\sum_{i=1}^n\sum_{g\in G}p_i^g\otimes q_i^{g^{-1}k}\right)\]
    Note that the $k$-th term in the right sum has degree $k$. Since $|1_S|=e$, it follows that these terms must vanish for
    $k\ne 1$. Then we have
            \[1_S=\alpha\left(\sum_{i=1}^n\sum_{g\in G}p_i^g\otimes q_i^{g^{-1}}\right)\]
        where the $p_i^g$ and $q_i^{g^{-1}}$ are homogeneous and $|q_i^{g^{-1}}|=g^{-1}=|p_i^g|^{-1}$. Thus, for the remainder of the proof, we assume $1_S = \alpha(\sum_i p_i \otimes q_i)$ with $p_i, q_i$ homogeneous and $|q_i| = |p_i|^{-1}$.

    Consider $T:=\{1,\dots,n\}$ as a $G$-graded set with grading given by $\xi:T\to G$, $\xi(i)=|p_i|$. The formula
    \[\o{\phi}(s)=\sum_{i,j}E_{i,j}\otimes \beta((q_i\cdot s)\otimes p_j)\]
    defines a graded homomorphism $\o{\phi}:S\to M_T(R)$. Let $\phi:M_G(S)\to M_nM_G(R)$ be the composite:
    \[\xymatrix@C=4em{M_G(S)\ar[r]^-{M_G(\o{\phi})} & M_G(M_T(R))\ar[r]^{\text{Lem. \ref{lem:CDAgraduada}}} & M_nM_G(R)}\]
    It is easily verified that $\phi(E_{g,h}\otimes s)=\sum_{i,j}E_{i,j}\otimes E_{gf(i),hf(j)}\otimes \beta((q_i\cdot s)\otimes p_j)$. Finally, let $\Phi:F(M_G(S))\to F(M_G(R))$ be the composite
    \[\xymatrix@C=4em{F(M_G(S))\ar[r]^-{F(\phi)} & F(M_nM_G(R))\ar[r]^-{F(\iota_n)^{-1}}_-{\cong} & F(M_G(R))}\]
    where $\iota_n$ is the upper-left corner inclusion. We will show that $\Phi$ is an isomorphism.

    By symmetry, we may find $\t{n}\in\N$, $\t{p}_1,\dots, \t{p}_\t{n}\in P$ and $\t{q}_1,\dots, \t{q}_\t{n}\in Q$ such that $\beta(\sum_{i=1}^{\t{n}}\t{q}_i\otimes\t{p}_i)=1_R$. We may as well assume that $\t{p}_i$ and $\t{q}_i$ are homogeneous with $|\t{p}_i|=|\t{q}_i|^{-1}$. Thus, we get a graded homomorphism $\psi:M_G(R)\to M_{\t{n}}M_G(S)$ defined by $\psi(E_{g,h}\otimes r)=\sum_{i,j}E_{i,j}\otimes E_{g\t{f}(i),h\t{f}(j)}\otimes \alpha((\t{p}_i\cdot r)\otimes \t{q}_j)$ and a morphism $\Psi:F(M_G(R))\to F(M_G(S))$ defined by $\Psi = F(\iota_n)^{-1}\circ F(\psi)$.

    Consider the following commutative diagram:
\[\xymatrix@C=5em{F(M_G(S))\ar@/_1pc/[dr]_-{\Phi}\ar[r]^-{\phi_*} & F(M_nM_G(R))\ar[r]^-{M_n(\psi)_*} & F(M_nM_\t{n}M_G(S))\\
& F(M_G(R))\ar@/_1pc/[dr]_-{\Psi}\ar[r]^-{\psi_*}\ar[u]^-{\cong}_-{(\iota_n)_*} & F(M_\t{n}M_G(S))\ar[u]^-{\cong}_-{(\iota_n)_*} \\
& & F(M_G(S))\ar[u]^-{\cong}_-{(\iota_\t{n})_*}
}\]
A direct computation shows that
$M_n(\psi)\circ \phi:M_G(S)\to M_n M_\t{n}M_G(S)$
is given by
\[E_{h,k}\otimes s\mapsto \sum_{i,j=1}^n\sum_{l,m=1}^\t{n}E_{i,j}\otimes E_{l,m}\otimes E_{h|p_i||\t{q}_l|, k|p_j||\t{q}_m|} \otimes \alpha(\t{p}_l\otimes q_i)s\alpha(p_j\otimes \t{q}_m).\]
Thus, upon applying $F$, the composite $M_n(\psi)\circ\phi$ induces the same morphism as the upper-left-corner inclusion by Lemma \ref{lem:Gcornergr}. This proves that $\Psi\circ\Phi=\id$. The fact that $\Phi\circ\Psi=\id$ follows by symmetry.
\end{proof}

\begin{coro}
    Let $F:\GrAlg\to\cC$ be a functor. If $F$ is $G$-graded stable, then $F$ is $G$-graded Morita invariant.
\end{coro}

\begin{proof}
    Let $R$ and $S$ be unital $G$-graded algebras. If $R$ and $S$ are $G$-graded Morita equivalent, we have
    \[F(R)\cong F(M_G(R))\cong F(M_G(S))\cong F(S).\]
    Indeed, the isomorphism in the middle holds by Proposition \ref{prop:iso_tensorMGgr} and the other two isomorphisms hold by $\Ggr$-stability.
\end{proof}

\appendix
\section{Noncommutative quotients}\label{sec:appendix}

Let $X$ be a simplicial set and let $v\in X_0$ be a vertex. The \emph{star of $v$ in $X$} is defined as the set $\St_X(v)$ consisting of all the simplices of $X$ that have $v$ as one of its vertices. We shall drop $X$ from the notation when it is clear from the context. Note that $\St(v)$ may not be a simplicial subset of $X$, so we write $\overline{\St(v)}$ for the simplicial subset of $X$ generated by $\St(v)$. Finally, for a simplicial subset $K\subseteq X$ we define
\[\St_X(K)=\bigcup_{v\in K_0}\St_X(v).\]

Let $G$ be a group acting on a simplicial set $X$. We will consider actions $G\curvearrowright X$ satisfying one of the following conditions, which may be considered combinatorial analogues of the notion of a properly discontinuous action on a topological space.
\begin{enumerate}[label=(\Alph*)]
    \item\label{item:A} For every $v\in X_0$ and every $g\in G$, $g\ne e$, $\St(v)\cap \St(g\cdot v)=\emptyset$.
    \item\label{item:B} For every $v\in X_0$ and every $g\in G$, $g\ne e$, $\St(v)\cap \St(\overline{\St(g\cdot v)})=\emptyset$.
\end{enumerate}

\subsection{Subdivision}\label{sec:subdivision}
Let $\sd(X)$ denote the barycentric subdivision of a simplicial set $X$ and consider the cosimplicial simplicial set $\sd:\Delta\to\sSet$, $[n]\mapsto \sd(\Delta^n)$. For any simplicial set $X$ there is an isomorphism $\sd(X)\cong X\otimes \sd$ in the sense of \cite{FP}*{Section 4.2}. Explicitly, the set of $k$-simplices of $X\otimes \sd$ is defined by
\[(X \otimes\sd)_k=\left(\coprod_n X_n\times \sd(\Delta^n)_k\right)/\sim\]
where $\sim$ is the equivalence relation generated by
\[(\alpha^*(\sigma), s)\sim (\sigma,\alpha_*(s))\]
for $\alpha\in\hom_\Delta([m],[n])$, $\sigma \in X_n$ and $s\in \sd(\Delta^m)_k$. Any $\beta\in\hom_\Delta([l],[k])$ induces a function $\beta^*:(X\otimes \sd)_k \to (X\otimes \sd)_l$ by the formula
\[\beta^*(\sigma, s)=(\sigma, \beta^*(s)).\]
These definitions make $X\otimes \sd$ into a simplicial set.

Recall that a $k$-simplex $s\in\sd(\Delta^n)_k$ is given by a family of non-empty subsets of $[n]$, $s=(S_0\subseteq S_1\subseteq \cdots \subseteq S_k)$. This simplex is called \emph{interior} if $S_k=[n]$, that is, if its last vertex is the barycenter of $\Delta^n$.
We will use the following result.

\begin{prop}[\cite{FP}*{Proposition 4.6.3 and Addendum 4.2.8}]\label{propo:FP} Let $X$ be a simplicial set and let $\tau\in \sd(X)_k$. Then the following assertions hold.
    \begin{enumerate}[label=(\roman*)]
        \item\label{item:FP1} The $k$-simplex $\tau$ has a \emph{unique} representation of the form $(\sigma, s)$, where $\sigma \in X_n$ is a non-degenerate simplex and $s\in \sd(\Delta^n)_k$ is an interior simplex. Such a pair $(\sigma, s)$ is called a \emph{minimal pair} and the integer $n$ is called the \emph{level} of $\tau$, $n=\lev(\tau)$.
        \item\label{item:FP2} Let $(\sigma, s)\in X_n\times \sd(\Delta^n)_k$ be the unique minimal pair representing $\tau$. If $(\tilde{\sigma}, \tilde{s})\in X_m\times \sd(\Delta^m)_k$ is any other pair representing $\tau$, then there exist a (not necessarily unique) face operator $\delta:[n]\to [m]$ and a (not necessarily unique) degeneracy operator $\rho:[m]\to [n]$ such that $\sigma=\delta^*(\tilde{\sigma})$ and $s=\rho_*(\tilde{s})$.
    \end{enumerate}
\end{prop}

An action of a group $G$ on a simplicial set $X$ induces an action of $G$ on $\sd(X)$. Under the identification $\sd(X)\cong X\otimes \sd$, the induced action is $g\cdot (\sigma, s)=(g\cdot \sigma, s)$, where $g\in G$, $\sigma \in X_n$ and $s\in \sd(\Delta^n)_k$.
The following lemma gathers some useful observations.

\begin{lem}\label{lem:level}
    Let $G$ be a group acting on a simplicial set $X$ and consider the induced action $G\curvearrowright \sd(X)$. Then the following assertions hold.
    \begin{enumerate}[label=(\alph*)]
        \item\label{item:mismolevel} We have $\lev(g\cdot \tau)=\lev(\tau)$ for every $g\in G$ and every simplex $\tau$ of $\sd(X)$.
        \item\label{item:1simp} If $\tau\in\sd(X)_1$ is a $1$-simplex with distinct vertices, then $\lev(d_0(\tau))=\lev(\tau)$ and $\lev(d_1(\tau))<\lev (\tau)$.
    \end{enumerate}
\end{lem}
\begin{proof}
    Let us prove \ref{item:mismolevel}. Write $\tau=(\sigma, s)$, with $(\sigma, s)$ a minimal pair. Then we have
    $g\cdot \tau=(g\cdot \sigma, s)$.
    Since $\sigma$ is non-degenerate, then $g\cdot \sigma$ is non degenerate too. Thus $(g\cdot \sigma, s)$ is the minimal pair representing $g\cdot \tau$ and $\lev(g\cdot \tau)=\lev(\tau)$.

    For \ref{item:1simp}, let $\tau\in\sd(X)_1$ with $d_0(\tau)\ne d_1(\tau)$. Let $(\sigma, s)\in X_n\times\sd(\Delta^n)_1$ be the unique minimal pair representing $\tau$. We have $s=(S_0\subseteq S_1)$ with $S_1=[n]$. Then:
    \begin{align*}
        d_0(\tau)&=(\sigma, d_0(s))=(\sigma, S_1)=(\sigma, [n]) \\
        d_1(\tau)&=(\sigma, d_1(s))=(\sigma, S_0)
    \end{align*}
    Since $\sigma$ is non-degenerate and $[n]\in\sd(\Delta^n)_0$ is an interior simplex, $(\sigma, [n])$ is the unique minimal pair representing $d_0(\tau)$. This implies that $\lev(d_0(\tau))=n=\lev(\tau)$. Since $d_0(\tau)\ne d_1(\tau)$, we have $S_0\subsetneq [n]$.
    By Proposition \ref{propo:FP} \ref{item:FP2} there exist a face operator $\delta: [m]\to [n]$ and a degeneracy operator $\rho:[n]\to [m]$ such that the minimal pair representing $d_1(\tau)$ is
    \[(\delta^*(\sigma), \rho_*(S_0))=(\delta^*(\sigma), [m])\in X_m\times \sd(\Delta^m)_0,\]
    where $m=\lev(d_1(\tau))$. On one hand, we have $|S_0|<n+1$ since $S_0\subsetneq [n]$. On the other, we have $m+1\leq |S_0|$ since $\rho(S_0)=[m]$. It follows that $\lev(d_1(\tau))=m<n=\lev(\tau)$.
\end{proof}

\begin{lem}\label{lem:sdfree}
    Let $G$ be a group acting freely on a simplicial set $X$. Then the induced action of $G$ on $\sd(X)$ is free.
\end{lem}
\begin{proof}
    Let $\tau\in\sd(X)_k$ and $g\in G$ be such that $g\cdot \tau=\tau$. Write $\tau=(\sigma, s)$, with $(\sigma, s)$ a minimal pair. We have
    \[(\sigma, s)=g\cdot (\sigma, s)=(g\cdot \sigma, s).\]
    Since $g\cdot \sigma$ is a non degenerate simplex of $X$, then $(g\cdot \sigma, s)$ is also a minimal pair. The uniqueness provided by Lemma \ref{propo:FP} implies that $g\cdot\sigma=\sigma$, and thus $g=e$, since $G$ acts freely on $X$.
\end{proof}

\begin{lem}\label{lem:freeimpliesA}
    Let $G$ be a finite group acting freely on a simplicial set $X$. Then the induced action of $G\curvearrowright\sd(X)$ satisfies the condition \ref{item:A}.
\end{lem}
\begin{proof}
     Let $v\in \sd(X)_0$ be a vertex and let $g\in G$, $g\ne e$. Note that $g\cdot v\ne v$ by Lemma \ref{lem:sdfree} and $\lev(g\cdot v)=\lev(v)$ by Lemma \ref{lem:level} \ref{item:mismolevel}. Suppose that $\St_{\sd(X)}(v)\cap\St_{\sd(X)}(g\cdot v)\ne \emptyset$. Then there is a $1$-simplex $\tau\in \sd(X)_1$ whose vertices are $v$ and $g\cdot v$. That is, $\tau$ is a $1$-simplex of $\sd(X)$ that has distinct vertices with the same level. This cannot happen by Lemma \ref{lem:level} \ref{item:1simp}.
    It follows that $\St_{\sd(X)}(v)\cap\St_{\sd(X)}(g\cdot v)= \emptyset$.
\end{proof}

\begin{lem}\label{lem:AimpliesB}
    Let $G$ be a finite group acting freely on a simplicial set $X$. If the action of $G$ on $X$ satisfies condition \ref{item:A}, then the induced action of $G\curvearrowright\sd(X)$ satisfies the condition \ref{item:B}.
\end{lem}
\begin{proof}
    Let $v\in \sd(X)_0$ be a vertex and let $g\in G$, $g\ne e$. Suppose that
    \begin{equation}\label{lem:conditionB}\St_{\sd(X)}(v)\cap \St_{\sd(X)}(\overline{\St_{\sd(X)}(g\cdot v)})\ne\emptyset.\end{equation}
    Unravelling the definitions, we may find $w\in\sd(X)_0$ and $\tau_1,\tau_2\in\sd(X)_1$ such that $v$ and $w$ are the vertices of $\tau_1$ and $w$ and $g\cdot v$ are the vertices of $\tau_2$. Note that $v$, $g\cdot v$ and $w$ are three distinct vertices of $\sd(X)$. Indeed, $g\cdot v\ne v$ since $G$ acts freely on $\sd(X)$ by Lemma \ref{lem:sdfree}. Moreover $w\ne v$ and $w\ne g\cdot v$ since $G\curvearrowright\sd(X)$ satisfies condition \ref{item:A}. We have four possible cases:
    \begin{enumerate}[label=(\roman*)]
        \item\label{item:casea} $d_0(\tau_1)=v$, $d_1(\tau_1)=w$, $d_0(\tau_2)=w$ and $d_1(\tau_2)=g\cdot v$;
        \item\label{item:caseb} $d_0(\tau_1)=w$, $d_1(\tau_1)=v$, $d_0(\tau_2)=w$ and $d_1(\tau_2)=g\cdot v$;
        \item\label{item:casec} $d_0(\tau_1)=v$, $d_1(\tau_1)=w$, $d_0(\tau_2)=g\cdot v$ and $d_1(\tau_2)=w$;
        \item\label{item:cased} $d_0(\tau_1)=w$, $d_1(\tau_1)=v$, $d_0(\tau_2)=g\cdot v$ and $d_1(\tau_2)=w$.
    \end{enumerate}
    The case \ref{item:cased} cannot happen since we would have $\lev(g\cdot v)=\lev (v)$ by Lemma \ref{lem:level} \ref{item:mismolevel} and $\lev(g\cdot v)>\lev(v)$ by Lemma \ref{lem:level} \ref{item:1simp}. The case \ref{item:casea} cannot happen since we would have $\lev(w)<\lev(v)=\lev(g\cdot v)<\lev(w)$,
    again by Lemma \ref{lem:level}. Suppose that we are in the case \ref{item:caseb}. Note that we have
    $\lev(\tau_1)=\lev(w)=\lev(\tau_2)=:n$
    by Lemma \ref{lem:level} \ref{item:1simp}.
    Represent $\tau_i$ by a minimal pair $(\sigma_i, s_i)$. Then we have
    \[(\sigma_1, [n])=(\sigma_1, d_0(s_1))=d_0(\tau_1)=w=d_0(\tau_2)=(\sigma_2, d_0(s_2))=(\sigma_2, [n]).\]
    Since all the pairs above are minimal, we have $\sigma_1=\sigma_2=:\sigma$. Put $m:=\lev(v)=\lev(g\cdot v)$. By Proposition \ref{propo:FP} \ref{item:FP2}, we can represent $v$ and $g\cdot v$ by minimal pairs as follows, where $\eta_1,\eta_2\in X_m$ are (proper) faces of $\sigma\in X_n$:
    \begin{align*}
        v&=(\eta_1, [m]) \\
        g\cdot v&=(\eta_2, [m])
    \end{align*}
    Since $g\cdot \eta_1$ is non-degenerate, $(g\cdot \eta_1, [m])$ is also a minimal pair representing $g\cdot v$ and we conclude that $g\cdot \eta_1=\eta_2$.
    Let $u\in X_0$ be any vertex of $\eta_1$. Then $g\cdot u$ is a vertex of $g\cdot \eta_1=\eta_2$. Since both $\eta_1$ and $\eta_2$ are faces of $\sigma$, 
    this means that $\sigma\in\St_X(u)\cap\St_X(g\cdot u)$, contradicting condition \ref{item:A}. We still have to prove that the case \ref{item:casec} cannot happen but, if we are in this case, we can easily reduce to the case \ref{item:caseb} after relabelling. Indeed, put:
    \begin{align*}
        \tilde{\tau}_1&:=\tau_2 \\
        \tilde{\tau}_2&:=g\cdot \tau_1 \\
        \tilde{v}&:=w \\
        \tilde{w}&:=g\cdot v
    \end{align*}
    With these new labels, the conditions of case \ref{item:caseb} are satisfied. Since none of the four conditions can happen, we conclude that \eqref{lem:conditionB} cannot hold.
\end{proof}

\begin{coro}\label{coro:freeimpliesB}
    Let $G$ be a finite group acting freely on a finite simplicial set $X$. Then the induced action $G\curvearrowright \sd^2(X)$ is free and satisfies condition \ref{item:B}.
\end{coro}
\begin{proof}
    It follows from lemmas \ref{lem:sdfree}, \ref{lem:freeimpliesA} and \ref{lem:AimpliesB}.
\end{proof}

\subsection{Quotient and crossed product} Let $\ell$ be a commutative ring with unit. There is a simplicial ring $\ell^\Delta$, $[n]\mapsto \ell^{\Delta^n}$, defined by
\[\ell^{\Delta^n}=\ell[t_0,\dots ,t_n]/\langle t_0+\cdots + t_n-1\rangle.\]
For a simplicial set $X$, put $\ell^X=\hom_{\sSet}(X,\ell^\Delta)$; see \cite{cortho}*{Section 3}. Then $\ell^X$ is an $\ell$-algebra and we have
\begin{equation}\label{eq:ellX}\ell^X=\lim_{\Delta^n\downarrow X}\ell^{\Delta^n},\end{equation}
where the limit is taken over the category of simplices of $X$. We think of $\ell^X$ as the algebra of piecewise polynomial functions on $X$. If $G$ is a group and $X$ is a $G$-simplicial set, then $\ell^X$ is a $G$-algebra.

\begin{lem}\label{lem:tentfunctions}
    Let $X$ be a finite simplicial set. Then there exists a family of piecewise polynomial functions $\{\phi_v\}_{v\in X_0}\subseteq\ell^X$ such that
    \begin{enumerate}[label=(\alph*)]
        \item $\phi_v(v)=1$ and $\phi_v(\sigma)=0$ if $\sigma\not\in \St(v)$;
        \item $\sum_{v \in X_0}\phi_v=1$.
    \end{enumerate}
\end{lem}
\begin{proof}
    An element of $\ell^X$ \eqref{eq:ellX} is a compatible family $(\phi_\sigma)_{\sigma}$ where $\sigma:\Delta^n\to X$ is a simplex and $\phi_\sigma\in\ell^{\Delta^ n}=\ell[t_0,\dots,t_n]/\langle 1-\sum_i t_i\rangle$. For each $v\in X_0$ we can define a function $\phi_v\in \ell^X$ by
\[(\phi_v)_\sigma=\sum_{i\in I(v,\sigma)}t_i,\]
where $I(v,\sigma)=\{0\leq i\leq n:\text{$v$ is the $i$-th vertex of $\sigma$}\}$. Here, the phrase \emph{$v$ is the $i$-th vertex of $\sigma$} means that the following diagram commutes:
\[\xymatrix@C=5em{\Delta^0\ar@/_1pc/[rr]_-{v}\ar[r]^{\text{$i$-th vertex}} & \Delta^n\ar[r]^{\sigma} & X}\]
It is verified that the $(\phi_v)_\sigma$ are compatible and indeed define $\phi_v\in \ell^X$. Moreover, $\phi_v$ vanishes outside $\St(v)$ and $\sum_{v\in X_0}\phi_v=1$.
\end{proof}

\begin{thm}\label{thm:mainapp}
    Let $G$ be a finite group acting freely on a finite simplicial set $X$. If the action $G\curvearrowright X$ satisfies condition \ref{item:B}, then the algebras $\ell^{X/G}$ and $\ell^{X}\rtimes G$ are Morita equivalent.  
\end{thm}

\begin{proof}
Put $R=\ell^{X/G}$ and $S=\ell^X$. Let $\pi:X\to X/G$ be the projection. Note that the inclusion $\pi^*:R\to S$ identifies $R$ with the subalgebra $S^G\subset S$ of those elements fixed by $G$. There is a natural algebra homomorphism $j:S\rtimes G\to \End_R(S)$ determined by $j(\phi\rtimes g)(\psi)=\phi(g\cdot \psi)$. We claim that $S$ is a finitely generated projective $R$ module and that $j$ is an isomorphism. By \cite{CHR}*{Theorem 1.3}, to prove this we can construct $\phi_1,\dots ,\phi_n, \psi_1,\dots, \psi_n\in S$ such that
\[\sum_{i=1}^n \phi_i(g\cdot \psi_i)=\begin{cases}
    1 & \text{if $g=e$,} \\
    0 & \text{otherwise.}
\end{cases}\]

Let $\{\phi_v\}_{v\in X_0}\subseteq \ell^X$ be the family of functions satisfying both conditions in Lemma \ref{lem:tentfunctions}.
By \cite{CE}*{Theorem 9.4.1}, for each $v\in X_0$ we may choose $\psi_v\in S$ such that $\psi_v|_{\overline{\St(v)}}\equiv 1$ and $\psi_v$ vanishes outside $\St(\overline{\St(v)})$. We claim that
\[\sum_{v\in X_0} \phi_v(g\cdot \psi_v)=\begin{cases}
    1 & \text{if $g=e$,} \\
    0 & \text{otherwise.}
\end{cases}\]
For $g=e$, this follows from the equality $\phi_v\psi_v=\phi_v$. For $g\ne e$, it follows from the fact that $\phi_v$ vanishes outside $\St(\overline{\St(v)})$, $g\cdot \psi_v$ vanishes outside $\St(\overline{\St(g\cdot v)})$ and these sets are disjoint.

We have proved that $S$ is a finitely generated projective $R$-module and that $S\rtimes G\cong \End_R(S)$. To show that $S\rtimes G$ and $R$ are Morita equivalent, it suffices to prove that the $R$-module $S$ is a generator; see for example \cite{lam}*{Thm. 18.24} or \cite{willie}*{Teorema 5.2.7}. To show that $S$ is a generator we can show that its trace is $R$; see for example \cite{lam}*{Thm. 18.11 and Thm. 2.44}. Consider $T\in \Hom_R(S,R)$, $T(\phi)=\sum_{g\in G}g\cdot \phi$. Choose a set of representatives of the orbits $F\subset X_0$ and put $\phi_F=\sum_{v\in F}\phi_v$. Then we have
\[T(\phi_F)=\sum_{g\in G}\sum_{v\in F}g\cdot \phi_v=\sum_{g\in G}\sum_{v\in F}\phi_{g\cdot v}=\sum_{v\in X_0}\phi_v=1.\]
Then the trace of $S$ is $R$ and the result follows.
\end{proof}

\begin{lem}[cf. Lemma 4.6.1]\label{lem:sdhtpyequiv}
    The morphism $\sd(\Delta^n)\to \Delta^0$ induces a (polynomial) homotopy equivalence $\ell\to\ell^{\sd(\Delta^n)}$.
\end{lem}
\begin{proof}
    Let $a:\ell\to\ell^{\sd(\Delta^n)}$ be induced by $\sd(\Delta^n)\to \Delta^0$; then $a$ is the inclusion of $\ell$ as constant functions on $\sd(\Delta^n)$. Let $b:\ell^{\sd(\Delta^n)}\to \ell$ be the evaluation at the barycenter of $\Delta^n$. We claim that $b$ is a homotopy inverse for $a$. We clearly have $b\circ a=\id$. We now proceed to define an elementary homotopy from $\id$ to $a\circ b$ by contracting $\sd(\Delta^n)$ to the barycenter of $\Delta^n$. Write $\Simp$ for the category of simplices of $\sd(\Delta^n)$ and $\IntSimp$ for its subcategory defined as follows. The objects of $\IntSimp$ are the \emph{interior} simplices of $\sd(\Delta^n)$, that is, those simplices whose last vertex is the barycenter of $\Delta^n$. A morphism from $s\in (\sd(\Delta^n))_k$ to $s'\in (\sd(\Delta^n))_{k'}$ is a non decreasing function $\alpha:[k]\to[k']$ making the triangle below commute and such that $\alpha(k)=k'$.
    \[\xymatrix{\Delta^k\ar[d]_-{\alpha}\ar[r]^-{s} & \sd(\Delta^n) \\
    \Delta^{k'}\ar@/_0.5pc/[ur]_-{s'} & }\]
    Using the description of the simplices of $\sd(\Delta^n)$ from Section \ref{sec:subdivision}, it is easily verified that $\IntSimp$ is a final subcategory of $\Simp$ in the sense of \cite{cwm}*{Section IX.3}.
    Indeed, let us fix any simplex $t:\Delta^l\to \sd(\Delta^n)$, $t=(T_0\subseteq \cdots \subseteq T_l)$, and show that the comma category $t\downarrow \IntSimp$ is non-empty. Define $\hat{t}:\Delta^{l+1}\to\sd(\Delta^n)$, $\hat{t}=(T_0\subseteq \cdots \subseteq T_l\subseteq [n])$, and let $\iota:[l]\to [l+1]$ be the inclusion. Then $\hat{t}$ is an interior simplex and $\iota:t\to \hat{t}$ is a morphism in $\Simp$. To show that $t\downarrow\IntSimp$ is connected, note that $\hat{t}$ is an initial object of $t\downarrow\IntSimp$.
    It follows that
    \[\sd(\Delta^n)=\colim_{s\in\Simp} \Delta^k=\colim_{s\in\IntSimp} \Delta^k.\]
    Since $A^-:\sSet\to\Alg^\op$ commutes with colimits for any $A\in\Alg$ \cite{cortho}*{Section 3}, we have:
    \begin{align*}\ell^{\sd(\Delta^n)} &=\lim_{s\in\IntSimp}\ell^{\Delta^k}\\
    (\ell^{\sd(\Delta^n)})[t]& = \lim_{s\in\IntSimp}\ell^{\Delta^k}[t].\end{align*}
    In the latter equality we are using moreover that $\ell^X[t]\cong (\ell[t])^X$ for a finite simplicial set $X$ \cite{cortho}*{Proposition 3.1.3}. Thus, to define a homotopy $H:\ell^{\sd(\Delta^n)}\to (\ell^{\sd(\Delta^n)})[t]$ it suffices to define a compatible family of homotopies $H_s:\ell^{\Delta^k}\to\ell^{\Delta^k}[t]$ for $s\in\IntSimp$. For an interior simplex $s:\Delta^k\to\sd(\Delta^n)$, define 
    \[H(t_i)=\begin{cases}
        (1-t)t_i & \text{for $i<k$,} \\
        t+(1-t)t_k & \text{for $i=k$.}
    \end{cases}\]
    Let $\alpha:[k]\to [k']$ be a morphism from $s$ to $s'$ in $\IntSimp$. Using that $\alpha(k)=k'$ it is straightforward to verify that $\alpha^*\circ H_{s'}=H_s\circ \alpha^*$. Thus, upon taking limit, these homotopies assemble into the desired homotopy from $\id$ to $a\circ b$.
\end{proof}

\begin{lem}\label{lem:sd}
    Let $G$ be a finite group acting on a finite simplicial set $X$. Then the last vertex map $\sd(X)\to X$ induces a $\kk^G$-equivalence $\ell^X\to\ell^{\sd(X)}$.
\end{lem}
\begin{proof}
    Let $\gamma_X:\sd(X)\to X$ denote the last vertex map.
    We first prove the result for $X=\Delta^n$ with the trivial $G$-action. It turns out that in this case, the morphism $\gamma_{\Delta^n}^*:\ell^{\Delta^n}\to\ell^{\sd(\Delta^n)}$ is actually a (polynomial) homotopy equivalence. For $n=0$ there is nothing to prove. For $n>0$, we have a commutative diagram of algebras as follows, where the diagonal morphisms are induced by the projections onto $\Delta^0$.
    \[\xymatrix{\ell^{\Delta^n}\ar[rr]^-{\gamma_{\Delta^n}^*} & & \ell^{\sd(\Delta^n)} \\
     & \ell\ar[ul]\ar[ur] &}\]
    Since both diagonal morphisms are homotopy equivalences (by Lemma \ref{lem:sdhtpyequiv}), the morphism $\gamma_{\Delta^n}^*$ is a homotopy equivalence as well.

Let us now prove the result for $X=G/H\times\Delta^n$, where $H\subseteq G$ is a subgroup, $\Delta^n$ has the trivial $G$ action and $X$ has the diagonal action. Since $X=\coprod_{G/H}\Delta^n$ and $\sd$ commutes with colimits, we have $\sd(X)=\coprod_{G/H}\sd(\Delta^n)$. Thus, we have:
\begin{align*}\ell^{X}&=\prod_{G/H}\ell^{\Delta^n}=\ell^{G/H}\otimes\ell^{\Delta^n}\\
\ell^{\sd(X)}&=\prod_{G/H}\ell^{\sd(\Delta^n)}=\ell^{G/H}\otimes\ell^{\sd(\Delta^n)}\end{align*}
Moreover, under these identifications, the morphism $\gamma_X^*$ becomes \[\id\otimes \gamma_{\Delta^n}^*:\ell^{G/H}\otimes\ell^{\Delta^n}\to \ell^{G/H}\otimes\ell^{\sd(\Delta^n)},\]
which is a homotopy equivalence since $\gamma_{\Delta^n}^*$ is.

    We now prove the result for general $X$, by induction on $n=\dim(X)$. If $n=0$, there is nothing to prove since the last vertex map is the identity morphism. Let $n\geq 1$ and suppose that we have proved the result for finite simplicial sets of dimension $k\leq n-1$. If $\dim(X)=n$, we have a pushout square of $G$-simplicial sets as follows, where $\{H_i\}$ is a finite family of subgroups of $G$.
    \begin{equation}\label{eq:sk}\begin{gathered}\xymatrix{\coprod_i G/H_i\times \partial\Delta^n\ar[r]\ar[d] & X_{n-1}\ar[d] \\
    \coprod_iG/H_i\times\Delta^n\ar[r] & X}\end{gathered}\end{equation}
    To ease notation, put $Y=\coprod_iG/H_i\times\Delta^n$ and $Z=\coprod_iG/H_i\times\partial\Delta^n$.
    Upon applying $\sd$ to \eqref{eq:sk} we get a second pushout square, and $\gamma$ is a morphism from the latter to \eqref{eq:sk}. By \cite{er}*{Lemma B.2 and Lemma B.5}, $\gamma$ induces a morphism of long exact Mayer-Vietoris sequences as follows, for every $E\in\Alg$:
    \[\xymatrix@C=1.2em{
     \ar[r] & \kk^G_*(\ell^{\sd(Z)},E)\ar[r]\ar[d]^-{\gamma_Z} & \kk^G_*(\ell^{\sd(Y)}, E)\oplus \kk^G_*(\ell^{\sd(X_{n-1})},E)\ar[r]\ar[d]^-{\gamma_Y\oplus \gamma_{X_{n-1}}} & \kk^G_*(\ell^{\sd(X)}, E)\ar[r]\ar[d]^-{\gamma_X} & \\
    \ar[r] & \kk^G_*(\ell^Z,E)\ar[r] & \kk^G_*(\ell^Y, E)\oplus \kk^G_*(\ell^{X_{n-1}},E)\ar[r] & \kk^G_*(\ell^X, E)\ar[r] &
    }\]
    The morphisms $\gamma_Z$ and $\gamma_{Y}\oplus\gamma_{X_{n-1}}$ are isomorphisms by the inductive hypothesis and the particular case that was proven above. It follows by the five lemma that $\gamma_X$ is an isomorphism as well.
\end{proof}

\begin{coro}\label{coro:kkeq}
    Let $G$ be a finite group acting freely on a finite simplicial set $X$. Then $\ell^X\rtimes G$ and $\ell^{X/G}$ are $\kk$-equivalent.
\end{coro}
\begin{proof}
    By Corollary \ref{coro:freeimpliesB}, the action $G\curvearrowright \sd^2(X)$ is free and satisfies condition \ref{item:B}. By Theorem \ref{thm:mainapp}, the algebras $\ell^{\sd^2(X)/G}$ and $\ell^{\sd^2(X)}\rtimes G$ are Morita equivalent. Thus, they are $\kk$-equivalent by Corollary \ref{cor:gid}. By Lemma \ref{lem:sd}, we have $\kk$-equivalences $\ell^{X/G}\sim_\kk \ell^{\sd^2(X/G)}=\ell^{\sd^2(X)/G}$ (since $\sd$ commutes with colimits) and $\ell^{X}\sim_{\kk^G}\ell^{\sd^2(X)}$. Finally, since the crossed product functor can be defined at the level of $\kk$-theory \cite{ekk}*{Prop. 5.1.2}, it follows that $\ell^X\rtimes G\sim_\kk \ell^{\sd^2(X)}\rtimes G$. Chaining these equivalences together yields
    \[\ell^{X/G} \sim_\kk \ell^{\sd^2(X)/G} \sim_\kk \ell^{\sd^2(X)}\rtimes G \sim_\kk \ell^X\rtimes G.\qedhere\]
\end{proof}

\end{document}